\documentclass[12pt,reqno]{amsart}
\usepackage{graphicx}
\usepackage{float}
\usepackage[caption = false]{subfig}
\usepackage{enumerate}
\usepackage{mathtools}
\usepackage{breqn}
\usepackage{array, geometry, graphicx}
\usepackage{amsmath,amsfonts,amssymb,amsthm,cite,mathrsfs}
\newtheorem{theorem}{Theorem}[section]
\newtheorem{corollary}[theorem]{Corollary}
\newtheorem{lemma}[theorem]{Lemma}

\theoremstyle{definition}

\theoremstyle{remark}
\newtheorem{remark}[theorem]{Remark}
\numberwithin{equation}{section}
\DeclareMathOperator{\RE}{Re}
\DeclareMathOperator{\IM}{Im}

\makeatletter
\@namedef{subjclassname@2010}{%
  \textup{2010} Mathematics Subject Classification}
\makeatother

\allowdisplaybreaks

\begin{document}

\title{On a Generalized Briot-Bouquet type Differential Subordination}
\author[S. Sivaprasad Kumar]{S. Sivaprasad Kumar}

\address{Department of Applied Mathematics, Delhi Technological University,
Delhi--110 042, India}
\email{spkumar@dce.ac.in}
\author[P. Goel]{Priyanka Goel$^*$}

\address{Department of Applied Mathematics, Delhi Technological University,
Delhi--110 042, India}
\email{priyanka.goel0707@gmail.com}
\begin{abstract}
We introduce and study the following special type of differential subordination implication:
\begin{equation}\label{abs}
	p(z)Q(z)+\frac{zp'(z)}{\beta p(z)+\alpha}\prec h(z)\quad\Rightarrow p(z)\prec h(z),
\end{equation}
which generalizes the Briot-Bouquet differential subordination, where $Q(z)$ is analytic and $0\neq\beta,\alpha\in\mathbb{C}.$ Further, some special cases of our result are also discussed. Finally, analogues of open door lemma and integral existence theorem with applications to univalent functions are obtained.
\end{abstract}



\maketitle
\let\thefootnote\relax\footnotetext{*Corresponding author\\The work presented here was supported by the Council of Scientific and Industrial Research(CSIR). Ref.No.:08/133(0018)/2017-EMR-I.}

\section{Introduction}
Let $\mathcal{H}$ be the class of analytic functions defined on the open unit disc $\mathbb{D}:=\{z\in\mathbb{C}:|z|<1\}.$ For any positive integer $n$ and complex number $a,$  $\mathcal{H}[a,n]$ denotes the subclass of $\mathcal{H}$ consisting of functions of the form $f(z)=a+a_nz^n+a_{n+1}z^{n+1}+\cdots.$  Let $\mathcal{A}_n$ be the class of functions of the form $f(z)=z+a_{n+1}z^{n+1}+\cdots$ and denote $\mathcal{A}:=\mathcal{A}_1.$ The subclass of $\mathcal{A}$ consisting of univalent functions is denoted by $\mathcal{S}$. A function $f\in\mathcal{A}$ is said to be typically real in $\mathbb{D}$ if for every non-real $z$ in $\mathbb{D},$ we have $sign(\IM f(z))=sign(\IM z).$ The class of all such functions is denoted by $\mathcal{TR}$. Let $\mathcal{P}$ be the class of analytic functions of the form $p(z)=1+c_1z+c_2z^2+\cdots$ such that $\RE p(z)>0$ for all $z$ in $\mathbb{D},$ $\mathcal{P}$ is known as the Carath\'eodory class. Given two analytic functions $f$ and $F,$ we say that $f$ is subordinate to $F,$ denoted by $f\prec F$ if there exists an analytic function $\omega(z),$ with $\omega(0)=0$ and $|\omega(z)|<1,$ such that $f(z)=F(\omega(z)).$ In particular, if $F$ is univalent then $f\prec F$ if and only if $f(0)=F(0)$ and $f(\mathbb{D})\subset F(\mathbb{D}).$ Assume $\mathcal{R}$ to be the class of functions with bounded turning consisting of all the functions $f\in\mathcal{A}$ such that $\RE f'(z)>0\;(z\in\mathbb{D}),$ clearly $\mathcal{R}\subset\mathcal{S}.$ A function $f$ in $\mathcal{S}$ is said to be starlike if and only if $\RE(zf'(z)/f(z))>0$ in $\mathbb{D}$ and the class of starlike functions is denoted by $\mathcal{S}^*.$ Similarly, the class of convex functions, denoted by $\mathcal{C}$ consists of all those functions $f$ in $\mathcal{S}$ for which $\RE(1+zf''(z)/f'(z))>0$ in $\mathbb{D}.$ Let $\mathcal{S}^*(\alpha)$ $(0\leq\alpha<1)$ be the subclass of $\mathcal{S}^*$ consisting of the functions $f$ satisfying $\RE(zf'(z)/f(z))>\alpha.$ We say that a function $f$ is strongly starlike of order $\gamma\;(0<\gamma\leq 1),$ whenever $|\arg{(zf'(z)/f(z))}|<\gamma\pi/2$ and the class of strongly starlike functions is denoted by $\mathcal{SS}^*(\gamma),$ note that $\mathcal{SS}^*(1)=\mathcal{S}^*(0)=\mathcal{S}^*.$ Coman~\cite{bbref6} defined that a function $f\in\mathcal{A}$ is said to be almost strongly starlike of order $\alpha,\alpha\in(0,1],$ with respect to the function $g\in\mathcal{S}^*(1-\alpha)$ if
\begin{equation*}
	\dfrac{g(z)f'(z)}{g'(z)f(z)}\prec\left(\dfrac{1+z}{1-z}\right)^{\alpha}\;\text{or equivalently,}\;\left|\arg{\dfrac{g(z)f'(z)}{g'(z)f(z)}}\right|<\alpha\dfrac{\pi}{2}
\end{equation*}
and concluded that such functions are starlike and hence univalent. Recently, Antonino and Miller~\cite{antoF} defined the class of $F$-starlike functions, denoted by $\mathcal{FS}^*$ as
\begin{equation*}
	\mathcal{FS}^*=\left\{f\in\mathcal{A}:\RE\left(\dfrac{F(z)f'(z)}{F'(z)f(z)}\right)>0\right\},
\end{equation*}
where $F$ is fixed univalent function on the closed unit
disk $\mathbb{D}$, with at most a single pole on $\partial\mathbb{D}$ and $F(0) = 0$. The authors in~\cite{maminda} came up with a generalized subclass of $\mathcal{S}^*$ as well as $\mathcal{C},$ defined using subordination. They considered an analytic function $\varphi,$ having positive real part, which is starlike with respect to $\varphi(0)=1$, symmetric about the real axis and satisfies $\varphi'(0)>0$ on $\mathbb{D}.$  Various subclasses of $\mathcal{S}^*$ are evolved for different choices of $\varphi$ in due course of time ( see~\cite{cho,first,mendiratta,crescent,sokolradius}). Note that $\varphi$ is a typically real function as $\varphi'(0)>0.$ Miller and Mocanu\cite{ds} introduced the theory of differential subordination as an analogue to differential inequalities, using the concept of subordination. The gist of this whole theory is the following implication:
\begin{equation}\label{aa}
	\psi(p(z),zp'(z),z^2p''(z))\prec h(z)\Rightarrow p(z)\prec q(z),\qquad z\in\mathbb{D},
\end{equation}
where $\psi(p(z),zp'(z),z^2p''(z))$ is analytic in $\mathbb{D}.$ Given a complex function $\psi:\mathbb{C}^3\times\mathbb{D}\rightarrow\mathbb{C}$ and a function $h$, univalent in $\mathbb{D},$ if $p$ is analytic in $\mathbb{D}$ and satisfies the differential subordination
\begin{equation}\label{a}
	\psi(p(z),zp'(z),z^2p''(z))\prec h(z),
\end{equation}
then $p$ is called a solution of~\eqref{a}. The univalent function $q$ is called a dominant of the solutions of the differential subordination, if $p\prec q$ for all $p$ satisfying~\eqref{a}. A dominant $\tilde{q}$ that satisfies $\tilde{q}\prec q$ for all dominants $q$ is said to be best dominant of~\eqref{a}. The implication~\eqref{aa} gives rise to three types of differential subordination problems stated in~\cite[Chapter.2]{ds} and  a good deal of literature associated with them is available (see~\cite{ssp2,vkt,first,pri3,ssp}). Using this theory, a special type of first order differential subordination, known as Briot-Bouquet differential subordination defined by
\begin{equation}\label{bb}
	p(z)+\dfrac{zp'(z)}{\beta p(z)+\alpha}\prec h(z),
\end{equation}
was studied by Miller and Mocanu~\cite{ds}. Many implication results were proved later associating~\eqref{bb}. Ruscheweyh and Singh~\cite{ds356} considered Briot-Bouquet differential subordination in a more particular form given by
\begin{equation*}
	p(z)+\dfrac{zp'(z)}{\beta p(z)+\alpha}\prec \dfrac{1+z}{1-z}
\end{equation*}
with $\alpha\geq 0$ and $\beta>0.$ Later it was generalized to the form given by~\eqref{bb}, in which $h(z)$ is taken to be a univalent function and $\alpha,\beta\neq0$ are extended to complex numbers. This particular differential subordination has vast number of applications in the univalent function theory, see~\cite{bbref3,bbref6,bbref4,bbref5,bbref2,bbref1} and the references therein. It is known that the Briot-Bouquet differential subordination is obtained from the Bernardi integral operator. Similarly, the general form of the Bernardi integral operator given by
\begin{equation}
		F(z)=I[f,g]=\left(\dfrac{\alpha+\beta}{g^{\alpha}(z)}\int_0^zg'(t)g^{\alpha-1}(t)f^{\beta}(t)dt\right)^{1/\beta}
\end{equation}
with appropriate choice of $p$ and $h$ yields a different type of differential subordination(for example see Corollary~\ref{corF}), which we introduce here:
\begin{equation}\label{1201}
	p(z)Q(z)+\frac{zp'(z)}{\beta p(z)+\alpha}\prec h(z)\quad (z\in\mathbb{D}),
\end{equation}
 where $\alpha,\;\beta\in\mathbb{C}$ with $\beta\neq 0$ and $Q$ is an analytic function such that
 	\begin{equation}
		g(z)=z\exp{\int_0^z \dfrac{Q(t)-1}{t}dt}.
	\end{equation}
The expression~\eqref{1201} is clearly a generalization of the Briot-Bouquet differential subordination as it is evident when we choose $Q(z)=1$. In the present investigation, we find conditions on $\alpha,\;\beta$ and $Q(z)$ so that the implication~\eqref{abs} holds. Further, we establish certain subordination results analogous to open door lemma and integral existence theorem. Apart from deriving other similar results, we also find sufficient conditions for starlikeness and univalence as an application of our results. For a better understanding of our main results, one may refer to~\cite[Ch.2]{ds} for the prerequisites.
\section{Generalized Briot-Bouquet Differential Subordination}
We present here all implication results pertaining to the proposed generalized Briot-Bouquet differential subordination. We begin with the following result:
\begin{theorem}\label{bbgen}
	Let $h$ be convex in $\mathbb{D}$ and $\alpha,\beta\in\mathbb{C}$ with $\beta\neq 0$. If $Q\in\mathcal{H}[1,n]$ be such that the following conditions hold:
	\begin{enumerate}
		\item[$(i)$] $\RE\left(\dfrac{1}{\beta h(z)+\alpha}\right)>0\quad (z\in\mathbb{D}).$
		\item[$(ii)$] $\RE\left(\dfrac{1}{\beta h(\zeta)+\alpha}+(Q(z)-1)\dfrac{h(\zeta)}{\zeta h'(\zeta)}\right)>0\quad (z\in\mathbb{D},\;\;\zeta\in h^{-1}(p(D))),$
	\end{enumerate}
	where $D=\{z\in\mathbb{D}:p(z)=h(\zeta)\;\text{for some}\; \zeta\in\partial\mathbb{D}\}.$ If $p$ is analytic in $\mathbb{D}$ with $p(0)=h(0)$ and
	\begin{equation}\label{0v}
		p(z)Q(z)+\dfrac{zp'(z)}{\beta p(z)+\alpha}\prec h(z),
	\end{equation}
	then $p(z)\prec h(z).$
\end{theorem}
\begin{proof}
	Let us suppose $p$ is not subordinate to $h$. Then by~\cite[Lemma 2.2d, pp.24]{ds} there exists $z_0\in\mathbb{D},$ $\zeta_0\in\partial\mathbb{D}$ and $m\geq1$ such that $p(z_0)=h(\zeta_0)$ and $z_0p'(z_0)=m\zeta_0h'(\zeta_0)$ and therefore we have
	\begin{equation*}
		\psi_0:=\psi(p(z_0),z_0p'(z_0))=\psi(h(\zeta_0),m\zeta_0 h'(\zeta_0))=h(\zeta_0)Q(z_0)+\dfrac{m\zeta_0 h'(\zeta_0)}{\beta h(\zeta_0)+\alpha},
	\end{equation*}
	which yields
	\begin{equation*}
		\RE\dfrac{\psi_0-h(\zeta_0)}{\zeta_0 h'(\zeta_0)}=\RE\left((Q(z_0)-1)\dfrac{h(\zeta_0)}{\zeta_0 h'(\zeta_0)}+\dfrac{m}{\beta h(\zeta_0)+\alpha}\right).
	\end{equation*}
	Using the fact that $m\geq1$ together with (i) and (ii), we have
	\begin{eqnarray*}
		\RE\dfrac{\psi_0-h(\zeta_0)}{\zeta_0 h'(\zeta_0)}&\geq& \RE\left((Q(z_0)-1)\dfrac{h(\zeta_0)}{\zeta_0 h'(\zeta_0)}+\dfrac{1}{\beta h(\zeta_0)+\alpha}\right)>0,
	\end{eqnarray*}
	which implies
	\begin{equation*}
		\left|\arg{\dfrac{\psi_0-h(\zeta_0)}{\zeta_0 h'(\zeta_0)}}\right|<\dfrac{\pi}{2}.
	\end{equation*}
	Since $h(\mathbb{D})$ is convex, $h(\zeta_0)\in h(\partial\mathbb{D})$ and $\zeta_0 h'(\zeta_0)$ is the outward normal to $h(\partial{\mathbb{D}})$ at $h(\zeta_0),$ we conclude that $\psi_0\notin h(\mathbb{D}),$ which contradicts~\eqref{0v} and hence $p(z)\prec h(z).$
\end{proof}
\begin{remark}
	If we take $Q(z)=1$ in Theorem~\ref{bbgen}, it reduces to~\cite[Theorem~3.2a]{ds}.
\end{remark}
\begin{corollary}\label{kcor}
	Let $h$ be convex in $\mathbb{D}$ and $\alpha,\beta\in\mathbb{C}$ with $\beta\neq 0$. If $Q\in\mathcal{H}[1,n]$ and $p$ is analytic in $\mathbb{D}$ with $p(0)=h(0)=(k-1)/4,$ where $k\geq1,$ be such that
	\begin{equation}\label{09}
		\RE\left(\dfrac{1}{\beta h(\zeta)+\alpha}\right)>k|Q(z)-1|-\RE(Q(z)-1)\quad (z\in\mathbb{D},\;\zeta\in\partial\mathbb{D}),
	\end{equation}
	then
	\begin{equation*}
		p(z)Q(z)+\dfrac{zp'(z)}{\beta p(z)+\alpha}\prec h(z)\quad\Rightarrow\quad p(z)\prec h(z).
	\end{equation*}
\end{corollary}
\begin{proof}
	Since $k\geq 1,$ from~\eqref{09} it is clear that for $\zeta\in\partial\mathbb{D},$
	\begin{equation}\label{01}
		\RE\left(\dfrac{1}{\beta h(\zeta)+\alpha}\right)>0.
	\end{equation}
	Since $h$ is convex, the above inequality holds on $\mathbb{D}$ as well. Also, we can say that $\tilde{h}(z):=h(z)-h(0)\in\mathcal{C}.$ Using Marx Strohh\~{a}cker theorem~\cite{ds}, we have $\RE(\zeta \tilde{h}'(\zeta)/\tilde{h}(\zeta))>1/2,$ which is equivalent to
	\begin{equation*}
		\left|\dfrac{\tilde{h}(\zeta)}{\zeta \tilde{h}'(\zeta)}-1\right|\leq 1,
	\end{equation*}
	which implies
	\begin{equation}\label{dis}
		\left|\dfrac{h(\zeta)}{\zeta h'(\zeta)}-1\right|\leq 1+\dfrac{|h(0)|}{|h'(\zeta)|}.
	\end{equation}
	Since $\tilde{h}\in\mathcal{C},$ we have $|\tilde{h}'(z)|\geq 1/(1+r)^2$ on $|z|=r$~\cite[Theorem.9, pp~118]{goodman}. We know that $\zeta\in\partial\mathbb{D},$ so we have $|h'(\zeta)|=|\tilde{h}'(\zeta)|\geq 1/4.$ Thus~\eqref{dis} reduces to
	\begin{equation*}
		\left|\dfrac{h(\zeta)}{\zeta h'(\zeta)}-1\right|\leq k.
	\end{equation*}
	Note that if $X,\;Y\in\mathbb{C}$ and $|X-1|\leq K,$ then
	\begin{equation*}
		\RE(X.Y)=\RE Y+\RE Y(X-1))\geq \RE Y-|Y|K.
	\end{equation*}
	Using this inequality, we can say that
	\begin{eqnarray*}
		\RE\left((Q(z)-1)\dfrac{h(\zeta)}{\zeta h'(\zeta)}+\dfrac{1}{\beta h(\zeta)+\alpha}\right)&\geq & \RE(Q(z)-1)-k|Q(z)-1|\\
		& &+\RE\left(\dfrac{1}{\beta h(\zeta)+\alpha}\right),
	\end{eqnarray*}
	which by using~\eqref{09} implies
	\begin{equation}\label{dis2}
		\RE\left((Q(z)-1)\dfrac{h(\zeta)}{\zeta h'(\zeta)}+\dfrac{1}{\beta h(\zeta)+\alpha}\right)> 0.
	\end{equation}
	From~\eqref{01} and~\eqref{dis2}, we may conclude that the conditions (i) and (ii) of Theorem~\ref{bbgen} are satisfied and as its application, the result follows.
\end{proof}
\begin{corollary}
	Let $Q\in\mathcal{H}[1,1]$ be a function such that $|Q(z)|\leq M\;(z\in\mathbb{D})$ for some $M>0$ and $\alpha,\;\beta\in\mathbb{C}$ with $\beta\neq0.$ Suppose $p$ is analytic and $h$ is convex in $\mathbb{D}$ with $p(0)=h(0)=1$  such that
	\begin{equation}\label{6M}
		\RE\left(\dfrac{1}{\beta h(\zeta)+\alpha}\right)>6(M+1)\quad (\zeta\in\partial\mathbb{D}),
	\end{equation}
	then
	\begin{equation*}
		p(z)Q(z)+\dfrac{zp'(z)}{\beta p(z)+\alpha}\prec h(z)\quad\Rightarrow\quad p(z)\prec h(z).
	\end{equation*}
\end{corollary}
\begin{proof}
	We know that $-\RE(Q(z)-1)\leq |Q(z)-1|$ and since $p(0)=1,$ in view of Corollary~\ref{kcor}, we have $k=5.$ Thus
	\begin{equation}\label{0x}
		k|Q(z)-1|-\RE(Q(z)-1)\leq 6|Q(z)-1|\leq 6(|Q(z)|+1)\leq 6(M+1).
	\end{equation}
	Since~\eqref{09} holds due to~\eqref{6M} and~\eqref{0x} and therefore the result follows from Corollary~\ref{kcor}.
\end{proof}
\begin{corollary}\label{corF}
	Let $h$ be convex in $\mathbb{D}$ with $h(0)=1$ and $\alpha,\beta\in\mathbb{C}$ with $\beta\neq 0.$ Let $g\in\mathcal{A}$ be defined as
	\begin{equation}\label{aa1}
		g(z)=z\exp{\int_0^z \dfrac{Q(t)-1}{t}dt},
	\end{equation}
	such that $h$ and $Q$ satisfy the conditions $(i)$ and $(ii)$ of Theorem~\ref{bbgen}. If $f\in\mathcal{A}$ and $F$ is given by
	\begin{equation}\label{24}
		F(z)=I[f,g]=\left(\dfrac{\alpha+\beta}{g^{\alpha}(z)}\int_0^zg'(t)g^{\alpha-1}(t)f^{\beta}(t)dt\right)^{1/\beta},
	\end{equation}
	then
	\begin{equation*}
		\dfrac{zf'(z)}{f(z)}\prec h(z)\quad\Rightarrow\quad\dfrac{zF'(z)/F(z)}{zg'(z)/g(z)}\prec h(z).
	\end{equation*}
\end{corollary}
\begin{proof}
	From~\eqref{aa1}, we have $Q(z)=zg'(z)/g(z)$ and let us suppose $p(z)=zF'(z)/(Q(z)F(z)).$ Then by differentiating~\eqref{24} and appropriately replacing the expressions, we have
	\begin{equation*}
		p(z)Q(z)+\dfrac{zp'(z)}{\beta p(z)+\alpha}=\dfrac{zf'(z)}{f(z)}.
	\end{equation*}
	Since $zf'(z)/f(z)\prec h(z),$ the result now follows from Theorem~\ref{bbgen}.
\end{proof}

If we take $h(z)=((1+z)/(1-z))^{\gamma}$ with $\gamma\in(0,1]$ in Corollary~\ref{corF}, we obtain the following result.
\begin{corollary}
	Let $f\in\mathcal{A}$ and $g,Q$ and $F$ are as defined in~\eqref{aa1} and~\eqref{24} respectively such that $\RE(Q(z)-1)>1-\gamma$. Then
	\begin{equation*}
		f\in\mathcal{SS}^*(\gamma)\;\Rightarrow\;F\;\text{is almost strongly starlike of order}\;\gamma\;\text{w.r.t the function}\;g.
	\end{equation*}
\end{corollary}

If we take $h(z)=(1+z)/(1-z)$ in Corollary~\ref{corF}, we obtain the following result.
\begin{corollary}
	Let $f\in\mathcal{A}$ and $g,Q$ and $F$ are as defined in~\eqref{aa1} and~\eqref{24} respectively such that $f\in\mathcal{S}^*$ and $g$ is univalent on $\overline{\mathbb{D}},$ then $F$ is a $g$-starlike function.
\end{corollary}

Now we list some of the special cases of Theorem~\ref{bbgen} here below:
\begin{corollary}\label{ez}
	Let $Q\in\mathcal{H}[1,1]$ and $\alpha,\beta$ be non-negative real numbers with $\beta\neq 0$ such that $|Q(z)-1|<1/(\beta e+\alpha)$ on $\mathbb{D}$. Suppose $p$ is analytic with $p(0)=1,$ then
	\begin{equation*}
		p(z)Q(z)+\dfrac{zp'(z)}{\beta p(z)+\alpha}\prec e^z\quad\Rightarrow\quad p(z)\prec e^z.
	\end{equation*}
\end{corollary}
\begin{proof}
	If we take $h(z)=e^z,$ then for $\alpha\geq 0$ and $\beta>0,$ we have
	\begin{equation}\label{0a}
		\RE\left(\dfrac{1}{\beta e^z+\alpha}\right)\geq \dfrac{1}{\beta e+\alpha}>0.
	\end{equation}
	Further, we observe for $z\in\mathbb{D}$ and $\zeta\in\partial\mathbb{D},$
	\begin{equation}\label{0b}
		\RE\left(\dfrac{Q(z)-1}{\zeta}+\dfrac{1}{\beta e^z+\alpha}\right)\geq-|Q(z)-1|+\dfrac{1}{\beta e+\alpha}>0.
	\end{equation}
	From~\eqref{0a} and~\eqref{0b}, we may conclude that both the conditions of Theorem~\ref{bbgen} are satisfied
	and thus the result now follows from Theorem~\ref{bbgen}.
\end{proof}
\begin{corollary}
	Let $Q\in\mathcal{H}[1,1]$ and $\alpha\geq0,\beta>0$ be such that \begin{equation}\label{02}
		|Q(z)-1|<\RE(Q(z)-1)+\dfrac{1}{2(\sqrt{2}\beta+\alpha)}.
	\end{equation}
	Suppose $p$ is analytic with $p(0)=1,$ then
	\begin{equation*}
		p(z)Q(z)+\dfrac{zp'(z)}{\beta p(z)+\alpha}\prec\sqrt{1+z}\quad\Rightarrow\quad p(z)\prec \sqrt{1+z}.
	\end{equation*}
\end{corollary}
\begin{proof}
	Let $h(z)=\sqrt{1+z},$ then condition (i) of Theorem~\ref{bbgen} is satisfied clearly as for $\alpha\geq0$ and $\beta>0,$
	\begin{equation*}
		\RE\left(\dfrac{1}{\beta \sqrt{1+z}+\alpha}\right)\geq\dfrac{1}{\beta \sqrt{2}+\alpha}>0,\quad z\in\mathbb{D}.
	\end{equation*}
	Now for $z\in\mathbb{D}$ and $\zeta\in\partial\mathbb{D},$ we have from~\eqref{02}
	\begin{eqnarray*}
		\RE\left(2(Q(z)-1)\left(1+\dfrac{1}{\zeta}\right)+\dfrac{1}{\beta \sqrt{1+z}+\alpha}\right)&\geq& 2\RE(Q(z)-1)-2|Q(z)-1|\\
		& &+\dfrac{1}{\sqrt{2}\beta +\alpha}\\
		&>&0,
	\end{eqnarray*}
	which implies that condition (ii) of Theorem~\ref{bbgen} is satisfied. Thus the result follows at once from Theorem~\ref{bbgen}.
\end{proof}
\begin{corollary}\label{ss}
	Let $Q\in\mathcal{H}[1,1]$ be such that $Q'(0)>0$ and $\alpha,\beta$ be non-negative real numbers with $\beta\neq 0$. Suppose $p$ is analytic with $p(0)=1$ and $p'(0)>0,$ then for $0<\gamma\leq1,$
	\begin{equation*}
		p(z)Q(z)+\dfrac{zp'(z)}{\beta p(z)+\alpha}\prec\left(\dfrac{1+z}{1-z}\right)^{\gamma}\quad\Rightarrow\quad p(z)\prec \left(\dfrac{1+z}{1-z}\right)^{\gamma}.
	\end{equation*}
\end{corollary}
\begin{proof}
	Let $h(z)=((1+z)/(1-z))^{\gamma}.$ Then condition (i) of Theorem~\ref{bbgen} clearly holds. For condition (ii) to hold, we need to show that
	\begin{equation*}
		\RE\left((Q(z)-1)\dfrac{1-\zeta^2}{2\gamma \zeta}+\dfrac{1}{\beta\left(\tfrac{1+\zeta}{1-\zeta}\right)^{\gamma}+\alpha}\right)>0\quad (z\in\mathbb{D},\;\zeta\in\partial\mathbb{D}).
	\end{equation*}
	Let $Q(z)=1+a_1z+a_2z^2+\cdots$ and define $R(z)=Q(z)-1,$ then $R(0)=0.$ Since $Q'(0)>0,$ it is typically real and it is easy to conclude that $R(z)$ is typically real. We know that $\zeta\in h^{-1}(p(D)),$ where $D:=\{z\in\mathbb{D}:p(z)=h(\zeta)\;\text{for some}\;\zeta\in\mathbb{D}\}.$ Since $h(z)=((1+z)/(1-z))^{\gamma}$ is typically real and $p'(0)>0,$ we have $sign(\IM z)=sign(\IM\zeta).$ Now we consider
	{\small
		\begin{eqnarray*}
			 \RE\left((Q(z)-1)\tfrac{1-\zeta^2}{2\gamma \zeta}\right)&=&\dfrac{1}{2\gamma}\bigg(\RE(Q(z)-1)\RE\left(\tfrac{1-\zeta^2}{\zeta}\right)\\
			 & &-\IM(Q(z)-1)\IM\left(\tfrac{1-\zeta^2}{\zeta}\right)\bigg)\\
			&=&-\dfrac{1}{2\gamma}\IM(Q(z)-1)\IM\left(\dfrac{1-\zeta^2}{\zeta}\right).
	\end{eqnarray*}}
	Taking $\zeta=e^{i\theta}\;(0\leq\theta\leq 2\pi),$ we have
	\begin{equation*}
		-\IM\left(\dfrac{1-\zeta^2}{\zeta}\right)=2\sin{\theta}\;\;\begin{cases}
			>0,&\quad \theta\in(0,\pi),\\
			<0, &\quad \theta\in(\pi,2\pi).\\
		\end{cases}
	\end{equation*}
	Since $Q(z)-1$ is typically real, $sign(\IM(Q(z)-1))=sign(\IM z).$ Thus we have
	\begin{equation*}
		\RE\left((Q(z)-1)\dfrac{1-\zeta^2}{2\gamma \zeta}\right)=2\sin{\theta}\IM(Q(z)-1)\geq 0.
	\end{equation*}
	Also for $\alpha\geq 0$ and $\beta>0,$ we have  $\RE\left(\beta\left(\tfrac{1+z}{1-z}\right)^{\gamma}+\alpha\right)>0$ and therefore
	\begin{equation*}
		\RE\left((Q(z)-1)\dfrac{1-\zeta^2}{2\gamma \zeta}+\dfrac{1}{\beta\left(\tfrac{1+\zeta}{1-\zeta}\right)^{\gamma}+\alpha}\right)>0.
	\end{equation*}
	Therefore the result follows at once from Theorem~\ref{bbgen}.
\end{proof}

Now we use a different technique to prove the next two results, which demonstrates the similar implication for Janowski functions.
\begin{theorem}\label{janowski}
	Let $p(z)$ be analytic in $\mathbb{D}$ with $p(0)=1$ and $Q\in\mathcal{H}[1,n]$ be a function such that $|Q(z)|<M$ for some $M>0.$ Let $-1\leq B<A<1$ and $-1<E<D\leq 1$ satisfy
	\begin{eqnarray}\label{17}\nonumber
		(A-B)(1-A)(1+E)&>&(1+|A|)(\beta+\alpha+|\beta A+\alpha B|)((1+D)(1-B)\\
		& &+M(1+E)(1-A))
	\end{eqnarray}
	for some $\alpha$ and $\beta,$ where $\alpha+\beta>0.$ If
	\begin{equation*}
		p(z)Q(z)+\dfrac{zp'(z)}{\beta p(z)+\alpha}\prec\dfrac{1+Dz}{1+Ez},
	\end{equation*}
	then $p(z)\prec (1+Az)/(1+Bz).$
\end{theorem}
\begin{proof}
	Let us define $P(z)$ and $\omega(z)$ as
	\begin{equation*}
		P(z):=p(z)Q(z)+\dfrac{zp'(z)}{\beta p(z)+\alpha},\qquad \omega(z):=\dfrac{p(z)-1}{A-Bp(z)}.
	\end{equation*}
	Then $\omega(z)$ is meromorphic in $\mathbb{D}$ and $\omega(0)=0.$ By the definition of $P(z)$ and $\omega(z),$ we have
	\begin{equation*}
		P(z)=\dfrac{1+A\omega(z)}{1+B\omega(z)}Q(z)+\dfrac{(A-B)z\omega'(z)}{(1+B\omega(z))[\beta(1+A\omega(z))+\alpha(1+B\omega(z))]}.
	\end{equation*}
	Now we need to show that $|\omega(z)|<1$ in $\mathbb{D}.$ On the contrary, let us assume that there exists a point $z_0\in\mathbb{D}$ such that $$\max_{|z|\leq|z_0|}|\omega(z)|=|\omega(z_0)|=1.$$ Then by~\cite[Lemma 1.3, pp.28]{rush}, there exists $k\geq 1$ such that $z_0\omega'(z_0)=k\omega(z_0).$ Now by taking $\omega(z_0)=e^{i\theta}$ $(0\leq\theta\leq 2\pi),$ we have
	\begin{eqnarray*}
		\left|P(z_0)\right| &=& \left|\dfrac{1+A\omega(z_0)}{1+B\omega(z_0)}Q(z_0)+\dfrac{(A-B)k\omega(z_0)}{(1+B\omega(z_0))(\beta(1+A\omega(z_0))+\alpha(1+B\omega(z_0)))}\right| \\
		&\geq& \left|\dfrac{1+Ae^{i\theta}}{1+Be^{i\theta}}\right|\left(\left|\dfrac{(A-B)k}{(1+Ae^{i\theta})(\beta(1+Ae^{i\theta})+\alpha(1+Be^{i\theta}))}\right|-|Q(z_0)|\right).
	\end{eqnarray*}
	We know that for $p>0,$ $|p+qe^{i\theta}|^2=p^2+q^2+2pq\cos{\theta},$ attains its maximum at $\theta=0$ if $q>0$ and at $\theta=\pi$ if $q\leq 0$. So, $\max_{0\leq\theta\leq 2\pi}|p+qe^{i\theta}|^2=(p+|q|)^2.$ Thus
	\begin{eqnarray*}
		|P(z_0)| \geq\dfrac{1-A}{1-B}\left(\dfrac{(A-B)k}{(1+|A|)(\beta+\alpha+|\beta A+\alpha B|)}-M\right).
	\end{eqnarray*}
	Clearly the expression on the right hand side is an increasing function of $k$ and attains its minimum at $k=1.$ So
	\begin{equation*}
		|P(z_0)|\geq\dfrac{1-A}{1-B}\left(\dfrac{(A-B)}{(1+|A|)(\beta+\alpha+|\beta A+\alpha B|)}-M\right).
	\end{equation*}
	Now from~\eqref{17}, we have
	$$|P(z_0)|>\dfrac{1+D}{1+E},$$
	which contradicts the fact that $P(z)\prec(1+Dz)/(1+Ez)$ and hence completes the proof.
\end{proof}

Using Theorem~\ref{janowski}, we obtain the following corollary.
\begin{corollary}
	Let $p(z)$ be analytic in $\mathbb{D}$ with $p(0)=1$ and $Q$ be a function such that
	\begin{equation*}
		z\exp{\int_0^z \dfrac{Q(t)-1}{t}}dt\in\mathcal{S}^*(\phi).
	\end{equation*}
	Let $-1\leq B<A<1,$ $-1<E<D\leq 1$ and $\alpha,\beta$ be such that $\alpha+\beta>0.$ Then
	\begin{equation*}
		p(z)Q(z)+\dfrac{zp'(z)}{\beta p(z)+\alpha}\prec\dfrac{1+Dz}{1+Ez}\Rightarrow p(z)\prec \dfrac{1+Az}{1+Bz},
	\end{equation*}
	whenever any of the following cases hold:
	\begin{enumerate}
		\item[$(i)$] $\phi(z)=e^z$ and $(A-B)(1-A)(1+E)>(1+|A|)(\beta+\alpha+|\beta A+\alpha B|)((1+D)(1-B)+e(1+E)(1-A))$
		\item [$(ii)$]$\phi(z)=\sqrt{1+z}$ and $(A-B)(1-A)(1+E)>(1+|A|)(\beta+\alpha+|\beta A+\alpha B|)((1+D)(1-B)+\sqrt{2}(1+E)(1-A))$
		\item[$(iii)$] $\phi(z)=2/(1+e^{-z})$ and $(A-B)(1-A)(1+E)(1+e)>(1+|A|)(\beta+\alpha+|\beta A+\alpha B|)((1+D)(1-B)(1+e)+2e(1+E)(1-A))$
		\item [$(iv)$]$\phi(z)=1+ze^z$ and $(A-B)(1-A)(1+E)>(1+|A|)(\beta+\alpha+|\beta A+\alpha B|)((1+D)(1-B)+(1+e)(1+E)(1-A))$
		\item [$(v)$] $\phi(z)=z+\sqrt{1+z^2}$ and $(A-B)(1-A)(1+E)>(1+|A|)(\beta+\alpha+|\beta A+\alpha B|)((1+D)(1-B)+(1+\sqrt{2})(1+E)(1-A)).$
	\end{enumerate}
\end{corollary}

Since the proof of Theorem~\ref{janowski} loses its validity when $A=1$ or $E=-1,$ the theorem does not reduce to the case when $P(z)$ and $p(z),$ both are subordinate to $(1+az)/(1-z)$ for $0\leq a\leq1.$ This case, we handle in the following result with a different approach.
\begin{theorem}
	Let $\alpha\geq0,\;\beta>0,\;0\leq a\leq 1.$ Assume $p(z)$ to be analytic in $\mathbb{D}$ with $p(0)=1$ and $p'(0)>0.$ Also, let $Q(z)\in\mathcal{H}[1,1]$ be such that $Q'(0)>0$ and $\RE Q(z)<1.$ If $p$ satisfies
	\begin{equation}\label{23}
		p(z)Q(z)+\dfrac{zp'(z)}{\beta p(z)+\alpha}\prec\dfrac{1+az}{1-z},
	\end{equation}
	then $p(z)\prec(1+az)/(1-z).$
\end{theorem}
\begin{proof}
	We need to show that $p(z)\prec q(z)=(1+az)/(1-z).$ For if $p\nprec q$, then using~\cite[Lemma~2.2d, pp.24]{ds}, there exists a point $z_0=r_0e^{i\theta_0}\in\mathbb{D},$ $\zeta_0\in\partial\mathbb{D}\backslash\{1\}$ and $m\geq 1$ such that $p(z_0)=q(\zeta_0),$ $z_0p'(z_0)=m\zeta_0q'(\zeta_0)$ and $p(\mathbb{D}_{r_0})\subset q(\mathbb{D}).$ Since $\zeta_0$ is a boundary point, we may assume that $\zeta_0=e^{i\theta}$ for $\theta\in[0,2\pi].$ Then
	\begin{eqnarray*}
		p(z_0)Q(z_0)+\dfrac{z_0p'(z_0)}{\beta p(z_0)+\alpha}&=&q(\zeta_0)Q(z_0)+\dfrac{m\zeta_0 q'(\zeta_0)}{\beta q(\zeta_0)+\alpha}\\
		&=&\left(\dfrac{1-a}{2}+i\dfrac{1+a}{2}\cot{\theta/2}\right)Q(z_0)\\
		& &-\dfrac{m(a+1)e^{i\theta}}{(1-e^{i\theta})((\beta+\alpha)+(\beta a-\alpha)e^{i\theta})}.
	\end{eqnarray*}
	Taking $Q(z_0)=u+iv,$ we have
	{\small
		\begin{eqnarray*}
			\RE\left( p(z_0)Q(z_0)+\dfrac{z_0p'(z_0)}{\beta p(z_0)+\alpha}\right)&=&\dfrac{(1-a)u}{2}-\dfrac{(1+a)\cot{(\theta/2)}v}{2}\\
			& &-\dfrac{m(1+a)(\beta(1-a)+2\alpha)(1-\cos{\theta})}{(\beta(1-a)+2\alpha)^2(1-\cos{\theta})^2+(\beta(a+1)\sin{\theta})^2}.
	\end{eqnarray*}}
	Since $Q'(0)>0$, it is clear that $Q(z)$ is typically real and thus
	\begin{equation}\label{03}
		sign(v)=sign(\IM(Q(z_0)))=sign(\IM(z_0)).
	\end{equation}
	We also have $p'(0)>0$ and $q(z)=(1+az)/(1-z)$ is typically real, which implies
	{\small
		\begin{equation}\label{04}
			sign(\IM(z_0))=sign((\IM(p(z_0))))=sign((\IM(q(\zeta_0))))=sign(\IM(\zeta_0))=sign(\sin{\theta}).
	\end{equation}}
	From~\eqref{03} and~\eqref{04}, we obtain $sign(v)=sign(\sin{\theta}),$ which is sufficient to conclude that $v\cot{\theta/2}>0$ for $\theta\in[0,2\pi].$ Also we know that $\alpha\geq 0,\;\beta>0,\;m\geq 1$ and $0\leq a\leq 1.$ Therefore
	\begin{equation*}
		\RE\left( p(z_0)Q(z_0)+\dfrac{z_0p'(z_0)}{\beta p(z_0)+\alpha}\right)<\dfrac{(1-a)u}{2}<\dfrac{1-a}{2},
	\end{equation*}
	which contradicts ~\eqref{23} and hence the result.
\end{proof}
\section{Other Subordination Results and Applications}
We begin with the following analogue of open door lemma:
\begin{theorem}\label{odlbb}
	Let $p(z)$ be analytic in $\mathbb{D}$ with $p(0)=1$ and $Q(z)\in\mathcal{P}.$ Suppose that $\alpha\geq 0,\beta>0$ and $p$ satisfies
	\begin{equation}\label{hypo}
		p(z)Q(z)+\dfrac{zp'(z)}{\beta p(z)+\alpha}=1,
	\end{equation}
	then $\RE p(z)>0.$
\end{theorem}
\begin{proof}
	Define the analytic function $g$ by
	$$g(z)=z\exp{\int_0^z\dfrac{Q(t)-1}{t}}dt.$$
	Then one can verify that the function $p(z),$ given by
	\begin{equation}\label{pz}
		p(z)=\dfrac{g^{\alpha}(z)z^{\beta}}{\beta}\left(\int_0^z g^{\alpha-1}(t)g'(t)t^{\beta}\right)^{-1}-\dfrac{\alpha}{\beta}
	\end{equation}
	is analytic in $\mathbb{D}$ and $p\in\mathcal{H}[1,n].$ The logarithmic differentiation of~\eqref{pz} reveals that $p$ is a solution of the differential equation~\eqref{hypo}. We now use~\cite[Theorem 2.3i, pp.35]{ds} to prove that $\RE p(z)>0$. Let $\Omega=\{1\}$ and $\psi(r,s;z)=rQ(z)+s/(\beta r+\alpha),$ then~\eqref{hypo} can be written as
	$$\{\psi(p(z),zp'(z);z)|z\in\mathbb{D}\}\subset\Omega.$$
	In view of~\cite[Theorem 2.3i, pp.35]{ds}, it is sufficient to show that $\psi\in\Psi_n[\Omega,1],$ which means admissibility conditions defined for the class $\Psi_n[\Omega,1]$ are satisfied by $\psi$ or equivalently:
	\begin{equation}\label{1}
		\psi(\rho i,\sigma;z)=\dfrac{\sigma}{\beta\rho i+\alpha}+Q(z)\rho i\neq 1,
	\end{equation}
	where $\rho\in\mathbb{R},\;\sigma\leq-\dfrac{n}{2}(1+\rho^2),\;z\in\mathbb{D}$ and $n\geq1.$
	Suppose on the contrary if we assume~\eqref{1} is false, then there exists some values $\rho_0,\;\sigma_0$ and $z_0$ such that \begin{equation}\label{2}
		\dfrac{\sigma_0}{\beta\rho_0 i+\alpha}+i\rho_0Q(z_0)=1.
	\end{equation}
	If we let $Q(z)=S(z)+iT(z),$ then~\eqref{2} yields
	\begin{equation}\label{3}
		\dfrac{\sigma_0\alpha}{\alpha^2+\beta^2\rho_0^2}-\rho_0T(z_0)=1\quad\text{and}\quad-\dfrac{\beta\sigma_0\rho_0}{\alpha^2+\beta^2\rho_0^2}+\rho_0S(z_0)=0.
	\end{equation}
	Since $\sigma_0<0$ and $\alpha\geq0,$ we have $\rho_0\neq 0.$ Therefore from~\eqref{3}, we deduce that
	$$\RE Q(z_0)=S(z_0)=\dfrac{\beta\sigma_0}{\alpha^2+\beta^2\rho_0^2}\leq-\dfrac{\beta n(1+\rho_0^2)}{2(\alpha^2+\beta^2\rho_0^2)}<0,$$
	which contradicts the hypothesis and hence it follows that $\RE p(z)>0.$
\end{proof}
\begin{corollary}
	Let $f\in\mathcal{A}$ be such that $f(z)/z\neq 0$ in $\mathbb{D}$ and $f\in\mathcal{R}.$ Then $\RE f(z)/z>0.$
\end{corollary}
\begin{proof}
	Let $Q(z)=f'(z)$ then $Q\in\mathcal{P}$ as $f\in\mathcal{R}.$ Choose $p(z)=z/f(z),  \beta=1$ and $\alpha=0$, then clearly $p(0)=1$ and $p(z)$  satisfies~\eqref{hypo}. Now by an application of Theorem~\ref{odlbb}, it follows that $\RE p(z)>0$ and hence the result.
\end{proof}

Here below we consider an analogue of integral existence theorem:
\begin{theorem}
	Let $\varphi,\phi\in\mathcal{H}[1,n],$ with $\varphi(z)\phi(z)\neq 0$ in $\mathbb{D}.$ Let $\lambda,\eta,\gamma$ and $\delta$ be complex numbers with $\eta\neq 0,$ $\lambda+\delta=\eta+\gamma=1$ and $\alpha,\beta$ be non negative real numbers with $\beta\neq 0$. Let $g\in\mathcal{A}_n$ and suppose that
	\begin{equation}\label{Q}
		Q(z)=\lambda\dfrac{zg'(z)}{g(z)}+\dfrac{z\varphi'(z)}{\varphi(z)}+\delta\prec\dfrac{1+z}{1-z}.
	\end{equation}
	If $F$ is defined by
	\begin{equation}\label{F}
		F(z)=\left(\dfrac{\alpha+\beta}{\phi^{\beta}(z)g^{\lambda\alpha}(z)\varphi^{\alpha}(z)z^{\delta\alpha+\beta\gamma}}\int_0^z g^{\lambda\alpha}(t)\varphi^{\alpha}(t)t^{\beta+\delta\alpha-1}Q(t)dt\right)^{\tfrac{1}{\eta\beta}},
	\end{equation}
	then $F\in\mathcal{A}_n,F(z)/z\neq0$ and for $z\in\mathbb{D},$
	\begin{equation}\label{conc}
		\RE\left(\dfrac{\eta\dfrac{zF'(z)}{F(z)}+\dfrac{z\phi'(z)}{\phi(z)}+\gamma}{\lambda\dfrac{zg'(z)}{g(z)}+\dfrac{z\varphi'(z)}{\varphi(z)}+\delta}\right)>0.
	\end{equation}
\end{theorem}
\begin{proof}
	Let $p(z)$ be given by
	\begin{equation}\label{p}
		p(z)=\dfrac{1}{\beta}g^{\lambda\alpha}(z)\varphi^{\alpha}(z)z^{\beta+\delta\alpha}\left(\int_0^z g^{\lambda\alpha}(t)\varphi^{\alpha}(t)t^{\beta+\delta\alpha-1}Q(t)dt\right)^{-1}-\dfrac{\alpha}{\beta}.
	\end{equation}
	Then $$p(z)=1+p_nz^n+p_{n+1}z^{n+1}+\cdots $$
	is analytic in $\mathbb{D}$ and $p\in\mathcal{H}[1,n].$ The logarithmic differentiation of~\eqref{p} shows that $p(z)$ satisfies~\eqref{hypo} with $Q(z)$ as given in~\eqref{Q}. From hypothesis, we have $Q(z)\prec(1+z)/(1-z).$ Thus the hypothesis of the Theorem.~\ref{odlbb} is fulfilled by $p$ as well as $Q$ and hence it follows that $\RE p(z)>0.$ Using~\eqref{F} and~\eqref{p}, we get
	\begin{equation}\label{55}
		F(z)=\left(\dfrac{(\alpha+\beta)z^{\beta(1-\gamma)}}{\phi^{\beta}(z)(\beta p(z)+\alpha)}\right)^{\tfrac{1}{\eta\beta}}=z\left(\dfrac{\alpha+\beta}{\phi^{\beta}(z)(\beta p(z)+\alpha)}\right)^{\tfrac{1}{\eta\beta}}.
	\end{equation}
	Clearly the expression in the bracket is analytic and non-zero, so we deduce that $F\in\mathcal{A}_n$ and $F(z)/z\neq 0.$ By differentiating~\eqref{55} logarithmically, we obtain
	$$\eta\dfrac{zF'(z)}{F(z)}=1-\gamma-\dfrac{z\phi'(z)}{\phi(z)}-\dfrac{zp'(z)}{\beta p(z)+\gamma},$$
	which by using~\eqref{hypo} simplifies to
	$$\eta\dfrac{zF'(z)}{F(z)}+\dfrac{z\phi'(z)}{\phi(z)}+\gamma=p(z)Q(z).$$
	Therefore
	$$\RE\dfrac{1}{Q(z)}\left(\eta\dfrac{zF'(z)}{F(z)}+\dfrac{z\phi'(z)}{\phi(z)}+\gamma\right)>0,$$ which is equivalent to~\eqref{conc} and hence completes the proof.
\end{proof}

We now obtain the following special case of the above theorem when  $\delta=\gamma=0,$ $\lambda=\eta=1$ and $\varphi(z)=1=\phi(z).$
\begin{corollary}
	Let $g\in\mathcal{S}^*,$ $\alpha\geq0,\beta>0$ and $F$ is defined as
	\begin{equation*}
		F(z)=\dfrac{\alpha+\beta}{g^{\alpha}(z)}\int_0^z g'(t)g^{\alpha-1}(t)t^{\beta}dt,
	\end{equation*}
	then $F$ is almost strongly starlike with respect to $g$ and hence univalent.
\end{corollary}
\begin{theorem}
	Let $f,g\in\mathcal{A}_n$ and $\phi\in\mathcal{H}[1,n]$ with $\phi(0)=1$ and $\phi(z)\neq 0$ in $\mathbb{D}.$ Let $\beta,\alpha$ and $\sigma$ be complex numbers such that $\beta\neq0$ and $\RE(\beta h(z)+\alpha)>0,$ where $h$ is a convex function with $h(0)=1$ in $\mathbb{D}.$ Suppose
	\begin{equation}\label{4}
		\beta\dfrac{zf'(z)}{f(z)}+\sigma\dfrac{zg'(z)}{g(z)}\prec \beta h(z)+\sigma.
	\end{equation}
	If $F$ is defined by
	\begin{equation}\label{5}
		F(z)=\left(\dfrac{\beta+\alpha}{z^{\alpha}\phi(z)}\int_0^zf^{\beta}(t)g^{\sigma}(t)t^{\alpha-\sigma-1}dt\right)^{\tfrac{1}{\beta}},
	\end{equation}
	then $F(z)\in\mathcal{A}_n$ and
	\begin{equation}\label{6}
		\dfrac{zF'(z)}{F(z)}+\dfrac{1}{\beta}\dfrac{z\phi'(z)}{\phi(z)}\prec h(z).
	\end{equation}
\end{theorem}
\begin{proof}
	Let $p(z)$ be given by
	\begin{equation}\label{7}
		p(z)=\dfrac{1}{\beta}f^{\beta}(z)g^{\sigma}(z)z^{\alpha-\sigma}\left(\int_0^z f^{\beta}(t)g^{\sigma}(t)t^{\alpha-\sigma-1}dt\right)^{-1}-\dfrac{\alpha}{\beta}.
	\end{equation}
	Then $p(z)$ is analytic in $\mathbb{D}$ and $p(0)=1.$ Upon logarithmic differentiation of~\eqref{7}, we deduce that $p(z)$ satisfies
	\begin{equation}\label{8}
		p(z)+\dfrac{zp'(z)}{\beta p(z)+\alpha}=\dfrac{zf'(z)}{f(z)}+\dfrac{\sigma}{\beta}\left(\dfrac{zg'(z)}{g(z)}-1\right).
	\end{equation}
	From~\eqref{4} and~\eqref{8}, it can be concluded that
	\begin{equation*}
		p(z)+\dfrac{zp'(z)}{\beta p(z)+\alpha}\prec h(z).
	\end{equation*}
	By applying ~\cite[Theorem 3.2a, pp.81]{ds}, we obtain $p(z)\prec h(z).$ Substituting~\eqref{5} in~\eqref{7}, we obtain
	\begin{equation*}
		p(z)=\dfrac{1}{\beta}\left(\dfrac{(\alpha+\beta)f^{\beta}(z)g^{\sigma}(z)}{F^{\beta}(z)\phi(z)z^{\sigma}}-\alpha\right).
	\end{equation*}
	Differentiating logarithmically the following equation
	\begin{equation*}
		\beta p(z)+\alpha=\dfrac{(\alpha+\beta)f^{\beta}(z)g^{\sigma}(z)}{F^{\beta}(z)\phi(z)z^{\sigma}}
	\end{equation*}
	and using~\eqref{8}, we obtain
	\begin{equation*}
		p(z)=\dfrac{zF'(z)}{F(z)}+\dfrac{1}{\beta}\dfrac{z\phi'(z)}{\phi(z)}
	\end{equation*}
	and thus~\eqref{6} follows.
\end{proof}

We now derive another result analogous to open door lemma as follows:
\begin{lemma}\label{lema1}
	Let n be a positive integer and $\alpha,\beta$ be non-negative real numbers with $\beta\neq 0$. Let $Q\in\mathcal{H}[1,n]$ satisfy
	\begin{equation}\label{11}
		Q(z)\prec 1+z+\dfrac{nz}{\beta+\alpha(1+z)}\equiv h(z).
	\end{equation}
	If $p\in\mathcal{H}[1,n]$ satisfies the differential equation
	\begin{equation}\label{12}
		p(z)Q(z)+\dfrac{zp'(z)}{\beta p(z)+\alpha}=1,
	\end{equation}
	then $p(z)\prec 1/(1+z).$
\end{lemma}
\begin{proof}
	Let us set $q(z)=1/(1+z),$ then
	\begin{equation*}
		h(z)=\dfrac{1}{q(z)}-\dfrac{nzq'(z)}{q(z)(\beta q(z)+\alpha)}.
	\end{equation*}
	If $g(z)=z/(\beta+\alpha(1+z)),$ then we have
	\begin{equation*}
		\RE \dfrac{zg'(z)}{g(z)}=1-\RE\dfrac{\alpha z}{\beta+\alpha(1+z)}>1-\dfrac{\alpha}{\beta+2\alpha}=\dfrac{\beta+\alpha}{\beta+2\alpha}>0.
	\end{equation*}
	Since $g$ is starlike and
	\begin{eqnarray*}
		\RE\dfrac{zh'(z)}{g(z)}&=&\RE\left((\beta+\alpha(1+z))\left(1+\dfrac{n(\beta+\alpha)}{(\beta+\alpha(1+z))^2}\right)\right)\\
		&=& \beta+\alpha(1+\RE z)+n(\beta+\alpha)\RE\left(\dfrac{1}{\beta+\alpha(1+z)}\right)\\
		&\geq& \beta+\dfrac{n(\beta+\alpha)}{\beta+2\alpha}>0,
	\end{eqnarray*}
	we deduce that $h$ is close to convex and hence univalent in $\mathbb{D}.$ Now we consider the boundary curve of $h$ defined as
	\begin{equation*}
		w=h(e^{i\theta})=u(\theta)+iv(\theta),\quad \theta\in(-\pi,\pi).
	\end{equation*}
	Suppose
	\begin{eqnarray}\label{14}
		\nonumber r(\theta)&=&|h(e^{i\theta})|=|e^{-i\theta}h(e^{i\theta})|=\left|e^{-i\theta}+1+\dfrac{n}{\beta+\alpha(1+e^{i\theta})}\right|\\
		&=&\left|1+\cos{\theta}+\dfrac{n(\beta+\alpha(1+\cos{\theta}))}{d(\theta)}-i\left(\sin{\theta}+\dfrac{n\alpha\sin{\theta}}{d(\theta)}\right)\right|,
	\end{eqnarray}
	where $d(\theta)=(\beta+\alpha(1+\cos{\theta}))^2+(\alpha\sin{\theta})^2.$ After simplifying~\eqref{14}, we have
	\begin{eqnarray*}
		r(\theta) &=& \sqrt{2(1+\cos{\theta})+\dfrac{n^2+2n(2\alpha+\beta)(1+\cos{\theta})}{d(\theta)}}.
	\end{eqnarray*}
	By using~\eqref{11} and~\eqref{12}, we deduce that
	\begin{equation}\label{13}
		Q(z)=\dfrac{1}{p(z)}-\dfrac{zp'(z)}{p(z)(\beta p(z)+\alpha)}\prec h(z).
	\end{equation}
	On the contrary, if $p\nprec q$ then by \cite[Lemma~2.2d]{ds}, there exist points $z_0\in\mathbb{D}$ and $\zeta_0\in\partial\mathbb{D}$ and $m\geq n,$ such that $p(z_0)=q(\zeta_0)$ and $z_0p'(z_0)=m\zeta_0q'(\zeta_0).$ From~\eqref{13}, we have
	\begin{equation*}
		Q(z_0)=\dfrac{1}{q(\zeta_0)}-\dfrac{m\zeta_0q'(\zeta_0)}{q(\zeta_0)(\beta q(\zeta_0)+\alpha)}=1+\zeta_0+\dfrac{m\zeta_0}{\beta+\alpha(1+\zeta_0)}.
	\end{equation*}
	For $\zeta_0=e^{i\theta},$ we have
	\begin{equation*}
		|Q(z_0)|=\sqrt{2(1+\cos{\theta})+\dfrac{m^2+2m(2\alpha+\beta)(1+\cos{\theta})}{d(\theta)}}\geq r(\theta)\quad \theta\in(-\pi,\pi),
	\end{equation*}
	where $r(\theta)$ is given in~\eqref{14}. This implies that $Q(z_0)\notin h(\mathbb{D}),$ which is a contradiction and thus $p(z)\prec 1/(1+z).$
\end{proof}
\begin{theorem}\label{thm1}
	Let $n$ be a positive integer and $\alpha,\beta$ be non negative real numbers with $\beta\neq 0$. Let $f\in\mathcal{A}_n$ and $F=I_{\alpha,\beta}[f]$ be defined as
	\begin{equation}\label{15}
		I_{\alpha,\beta}[f]=\left(\dfrac{\alpha+\beta}{f^{\alpha}(z)}\int_0^zf^{\alpha-1}(t)f'(t)t^{\beta}dt\right)^{1/\beta}.
	\end{equation}
	If
	\begin{equation*}
		\dfrac{zf'(z)}{f(z)}\prec 1+z+\dfrac{nz}{\beta+\alpha(1+z)},
	\end{equation*}
	then $|zf'(z)/f(z)|<2|zF'(z)/F(z)|.$
\end{theorem}
\begin{proof}
	Let $f\in\mathcal{A}$ satisfy~\eqref{15} and define
	\begin{equation*}
		p(z)=z^{\beta}f^{\alpha}(z)\left(\beta\int_0^z f^{\alpha-1}(t)f'(t)t^{\beta}dt\right)^{-1}-\dfrac{\alpha}{\beta}.
	\end{equation*}
	By the series expansion, it is easy to verify that $p$ is well defined and $p\in\mathcal{H}[1,n].$ If we let $Q(z)=zf'(z)/f(z),$ then it is easy to show that $p$ satisfies~\eqref{12}. Hence by Lemma~\ref{lema1} we deduce that $p(z)\prec 1/(1+z).$ Since $p(z)\neq 0,$ we can define the analytic function $F\in\mathcal{A}_n$ by
	\begin{equation*}
		F(z)=z\left(\dfrac{\alpha+\beta}{\beta p(z)+\alpha}\right)^{1/\beta}.
	\end{equation*}
	A simple calculation shows that this function coincides with the function given in~\eqref{15}. So we obtain
	\begin{equation*}
		\dfrac{F(z)Q(z)}{zF'(z)}=\dfrac{1}{p(z)}\prec 1+z,
	\end{equation*}
	which further implies $\left|\tfrac{zf'(z)/f(z)}{zF'(z)/F(z)}-1\right|<1$ and hence the result follows at once.
\end{proof}
\begin{theorem}\label{thm2}
	Let n be a positive integer and $\alpha,\beta$ be non-negative real numbers with $\beta\neq 0$. Let $f\in\mathcal{A}$ satisfy
	\begin{equation*}
		\dfrac{f(z)}{zf'(z)}-\left(\dfrac{zf''(z)}{f'(z)}+1-\dfrac{zf'(z)}{f(z)}\right)\left(\beta\dfrac{zf'(z)}{f(z)}+\alpha\right)^{-1}\prec 1+z+\dfrac{nz}{\beta+\alpha(1+z)}.
	\end{equation*}
	then $zf'(z)/f(z)\prec 1/(1+z).$
\end{theorem}
\begin{proof}
	Let $p(z)=zf'(z)/f(z).$ Then
	\begin{equation*}
		Q(z)=\dfrac{f(z)}{zf'(z)}-\left(\dfrac{zf''(z)}{f'(z)}+1-\dfrac{zf'(z)}{f(z)}\right)\left(\beta\dfrac{zf'(z)}{f(z)}+\alpha\right)^{-1}
	\end{equation*}
	satisfies the following differential equation
	\begin{equation*}
		p(z)Q(z)+\dfrac{zp'(z)}{\beta p(z)+\alpha}=1.
	\end{equation*}
	Now applying Lemma~\ref{lema1}, we obtain $p(z)\prec 1/(1+z)$ and that completes the proof.
\end{proof}

If we take $\beta=1,\alpha=0$ and $n=1$ in the above theorem, we obtain the following corollary, which is a particular case of a result of Tuneski~\cite[Theorem 2.5]{tun2}.
\begin{corollary}
	Let $f\in\mathcal{A}$ satisfies
	\begin{equation*}
		\left| \dfrac{f(z)f''(z)}{(f'(z))^2}\right|<2,
	\end{equation*}
	then $zf'(z)/f(z)\prec 1/(1+z).$
\end{corollary}
\begin{theorem}\label{thm3}
	Let n be a positive integer and $\alpha,\beta$ be non-negative real numbers with $\beta\neq 0$. Let $f\in\mathcal{A}$ satisfy
	\begin{equation*}
		\dfrac{f(z)}{z}+\left(\dfrac{zf'(z)}{f(z)}-1\right)\left(\dfrac{\beta z}{f(z)}+\alpha\right)^{-1}\prec 1+z+\dfrac{nz}{\beta+\alpha(1+z)},
	\end{equation*}
	then $f(z)/z\prec 1+z.$
\end{theorem}
\begin{proof}
	Taking $p(z)=z/f(z)$ and $$Q(z)=\dfrac{f(z)}{z}+\left(\dfrac{zf'(z)}{f(z)}-1\right)\left(\dfrac{\beta z}{f(z)}+\alpha\right)^{-1},$$
	the proof goes similarly as that of Theorem~\ref{thm2} and the result follows at once.
\end{proof}

By taking $\beta=1, \alpha=0$ and $n=1$ in Theorem~\ref{thm3}, we obtain the following result.
\begin{corollary}
	Let $f\in\mathcal{A}$ be such that $f'(z)\prec 1+2z,$ then
	$f(z)/z\prec 1+z,$
	or equivalently
	$$|f'(z)-1|<2\Rightarrow |f(z)/z-1|<1.$$
\end{corollary}

We now prove the following lemma in order to derive some sufficient conditions for starlikeness:
\begin{lemma}\label{lema2}
	Let $n$ be a positive integer and $\alpha,\beta$ be non negative integers with $\beta\neq 0.$ Suppose that either $\alpha<\beta<3\alpha$ or $\beta<\alpha<3\beta$ and $Q\in\mathcal{H}[1,n]$ satisfy
	\begin{equation}\label{18}
		Q(z)\prec \dfrac{1+z}{1-z}+\dfrac{2nz}{(1-z)((\alpha+\beta)+(\alpha-\beta)z)}=h(z).
	\end{equation}
	If $p\in\mathcal{H}[1,n]$ satisfies the differential equation
\end{lemma}
\begin{equation}\label{19}
	p(z)Q(z)+\dfrac{zp'(z)}{\beta p(z)+\alpha}=1,
\end{equation}
then $p(z)\prec (1-z)/(1+z).$
\begin{proof}
	Let us set $q(z)=(1-z)/(1+z),$ then
	\begin{equation*}
		h(z)=\dfrac{1}{q(z)}-\dfrac{nzq'(z)}{q(z)(\beta q(z)+\alpha)}.
	\end{equation*}
	We know that the Koebe function $k(z)=z/(1-z)^2$ is starlike and
	\begin{eqnarray*}
		\RE\dfrac{zh'(z)}{k(z)}&=&\RE\left((1-z)^2\left(\dfrac{2}{(1-z)^2}+\dfrac{2n((\alpha+\beta)+(\alpha-\beta)z^2)}{(1-z)^2((\alpha+\beta)+(\alpha-\beta) z)^2}\right)\right)\\
		&=& \RE\left(2+\dfrac{2n((\alpha+\beta)+(\alpha-\beta)z^2)}{((\alpha+\beta)+(\alpha-\beta)z)^2}\right)\\
		&=& 2+2n\RE\left(\dfrac{(\alpha+\beta)+(\alpha-\beta)z^2}{((\alpha+\beta)+(\alpha-\beta)z)^2}\right).
	\end{eqnarray*}
	Now taking $\alpha+\beta=a,$ $\alpha-\beta=b$ and $z=e^{i\theta}$ for $\theta\in(-\pi,\pi),$ we have
	\begin{eqnarray*}
		\left|\arg\left(\dfrac{a+be^{2i\theta}}{(a+be^{i\theta})^2}\right)\right|&=&|\arg{(a+be^{2i\theta})}-2\arg{(a+be^{i\theta})}|\\
		&=&\left|\arctan{\left(\dfrac{b\sin{2\theta}}{a+b\cos{2\theta}}\right)}-2\arctan{\left(\dfrac{b\sin{\theta}}{a+b\cos{\theta}}\right)}\right|\\
		&\leq&\left|\arctan{\left(\dfrac{b\sin{2\theta}}{a+b\cos{2\theta}}\right)}\right|+2\left|\arctan{\left(\dfrac{b\sin{\theta}}{a+b\cos{\theta}}\right)}\right|.
	\end{eqnarray*}
	Since $\arctan{x}$ is an increasing function in $(-\pi,\pi)$, we now find the maximum of $g(x):=b\sin{x}/(a+b\cos{x})$ for $x\in(-\pi,\pi).$ Clearly, $a>0$ and after some elementary calculations, we deduce that $g(x)$ attains its maximum at $x=\pi-\arccos{(b/a)}$ when $b>0$ and at $x=-\arccos{(-b/a)}$ when $b<0,$ $b/\sqrt{a^2-b^2}$ and $-b/\sqrt{a^2-b^2}$ are the corresponding maximum values. Thus we have
	\begin{equation*}
		\left|\arg\left(\dfrac{a+be^{2i\theta}}{(a+be^{i\theta})^2}\right)\right|\leq 3\arctan{\left(\dfrac{|b|}{\sqrt{a^2-b^2}}\right)}.
	\end{equation*}
	For the case when $\beta<\alpha<3\beta,$ we observe that
	\begin{equation*}
		3\arctan{\left(\tfrac{|b|}{\sqrt{a^2-b^2}}\right)}= 3\arctan{\left(\tfrac{\alpha-\beta}{2\sqrt{\alpha\beta}}\right)}<3\arctan{\left(\tfrac{1}{3}\sqrt{\tfrac{\alpha}{\beta}}\right)}
		<3\arctan{\left(\tfrac{1}{\sqrt{3}}\right)}=\tfrac{\pi}{2}.
	\end{equation*}
	The other case can also be verified in the similar way and it can be concluded that
	\begin{equation*}
		\left|\arg\left(\dfrac{(\alpha+\beta)+(\alpha-\beta)z^2}{((\alpha+\beta)+(\alpha-\beta)z)^2}\right)\right|<\dfrac{\pi}{2},
	\end{equation*}
	which further implies that $\RE(zh'(z)/k(z))>0.$ So $h$ is close to convex and hence univalent in $\mathbb{D}.$ Now we consider the boundary curve of $h$ defined as
	\begin{equation*}
		h(e^{i\theta})=u(\theta)+iv(\theta),\quad \theta\in(-\pi,\pi).
	\end{equation*}
	Since $e^{i\theta}$ is a boundary point, without loss of generality we may assume that $(1+e^{i\theta})/(1-e^{i\theta})=i\gamma$ and thus
	\begin{eqnarray}
		\nonumber r(\theta)&=&|h(e^{i\theta})|=\left|i\gamma+\dfrac{2ne^{i\theta}}{(1-e^{2i\theta})(\beta/i\gamma+\alpha)}\right|=\left|i\gamma-\dfrac{n}{\sin{\theta}(\beta/\gamma+i\alpha)}\right|.
	\end{eqnarray}
	So
	\begin{eqnarray}\label{21}
		r(\theta)= \sqrt{\gamma^2+\dfrac{n^2+2n\alpha\gamma\sin{\theta}}{\sin^2{\theta}((\beta/\gamma)^2+\alpha^2)}}=\sqrt{\gamma^2+\dfrac{n^2+4n\alpha\gamma^2/(1+\gamma^2)}{\sin^2{\theta}((\beta/\gamma)^2+\alpha^2)}}.
	\end{eqnarray}
	From~\eqref{18} and~\eqref{19}, we deduce that
	\begin{equation}\label{20}
		Q(z)=\dfrac{1}{p(z)}-\dfrac{zp'(z)}{p(z)(\beta p(z)+\alpha)}\prec h(z).
	\end{equation}
	On the contrary if $p$ is not subordinate to $q$, then by \cite[Lemma.2.2d]{ds}, there exist points $z_0\in\mathbb{D}$ and $\zeta_0\in\partial\mathbb{D}$ and $m\geq n,$ such that $p(z_0)=q(\zeta_0)$ and $z_0p'(z_0)=m\zeta_0q'(\zeta_0).$ From~\eqref{20}, we have
	\begin{equation*}
		Q(z_0)=\dfrac{1}{q(\zeta_0)}-\dfrac{m\zeta_0q'(\zeta_0)}{q(\zeta_0)(\beta q(\zeta_0)+\alpha)}=\dfrac{1+\zeta_0}{1-\zeta_0}+\dfrac{2m\zeta_0}{(1-\zeta_0)((\beta+\alpha)+(\alpha-\beta)\zeta_0)}.
	\end{equation*}
	For $\zeta_0=e^{i\theta},$ we have
	\begin{equation*}
		|Q(z_0)|=\sqrt{\gamma^2+\dfrac{m^2+4m\alpha\gamma^2/(1+\gamma^2)}{\sin^2{\theta}((\beta/\gamma)^2+\alpha^2)}}\geq r(\theta)\quad \theta\in(-\pi,\pi),
	\end{equation*}
	where $r(\theta)$ is given in~\eqref{21}. This implies that $Q(z_0)\notin h(\mathbb{D}),$ which is a contradiction and thus we  have $p(z)\prec (1-z)/(1+z).$
\end{proof}
\begin{theorem}\label{noproof}
	Let $n$ be a positive integer and $\alpha,\beta$ be non negative real numbers with $\beta\neq 0,$ either $\alpha<\beta<3\alpha$ or $\beta<\alpha<3\beta.$ Let $f\in\mathcal{A}_n$ and $F=A_{\alpha,\beta}[f]$ is as defined in~\eqref{15}. If
	\begin{equation*}
		\dfrac{zf'(z)}{f(z)}\prec \dfrac{1+z}{1-z}+\dfrac{2nz}{(1-z)((\alpha+\beta)+(\alpha-\beta)z)},
	\end{equation*}
	then $\RE\left(\tfrac{zF'(z)/F(z)}{zf'(z)/f(z)}\right)>0.$
\end{theorem}
The proof of Theorem~\ref{noproof} follows by an application of Lemma~\ref{lema2}, similar to that of Theorem~\ref{thm1} and therefore it is omitted here.
\begin{theorem}\label{thm4}
	Let n be a positive integer and $\alpha,\beta$ be non-negative real numbers with $\beta\neq 0$. Suppose that either $\alpha<\beta<3\alpha$ or $\beta<\alpha<3\beta$ and $f\in\mathcal{A}$ satisfies
	\begin{equation*}
		\Theta(f)\prec \dfrac{1+z}{1-z}+\dfrac{2nz}{(1-z)((\alpha+\beta)+(\alpha-\beta)z)},
	\end{equation*}
	where
	\begin{equation*} \Theta(f)=\dfrac{f(z)}{zf'(z)}-\left(\dfrac{zf''(z)}{f'(z)}+1-\dfrac{zf'(z)}{f(z)}\right)\left(\beta\dfrac{zf'(z)}{f(z)}+\alpha\right)^{-1},
	\end{equation*}
	then $zf'(z)/f(z)\prec (1-z)/(1+z).$
\end{theorem}
The proof is omitted here as it is much akin to Theorem~\ref{thm2} and can be easily done by using Lemma~\ref{lema2}.
\begin{remark}
	If we take $n=1,$ $\beta=1$ and $\alpha=0$ in Theorem~\ref{thm4}, it reduces to a result of Tuneski~\cite[Theorem 2.1]{tun2}.
\end{remark}
\begin{theorem}
	Let n be a positive integer and $\alpha,\beta$ be non-negative real numbers with $\beta\neq 0$. Suppose that either $\alpha<\beta<3\alpha$ or $\beta<\alpha<3\beta$ and $f\in\mathcal{A}$ satisfies
	\begin{equation*}
		\dfrac{f(z)}{z}+\left(\dfrac{zf'(z)}{f(z)}-1\right)\left(\dfrac{\beta z}{f(z)}+\alpha\right)^{-1}\prec \dfrac{1+z}{1-z}+\dfrac{2nz}{(1-z)((\alpha+\beta)+(\alpha-\beta)z)},
	\end{equation*}
	then $f(z)/z\prec (1+z)/(1-z).$
\end{theorem}
\begin{proof}
	Taking $p(z)=z/f(z)$ and $$Q(z)=\dfrac{f(z)}{z}+\left(\dfrac{zf'(z)}{f(z)}-1\right)\left(\dfrac{\beta z}{f(z)}+\alpha\right)^{-1},$$
	the result follows by an application of Lemma~\ref{lema2}.
\end{proof}
By taking $\beta=1, \alpha=0$ and $n=1,$ we obtain the following result
\begin{corollary}
	Let $f\in\mathcal{A}$ be such that
	$$f'(z)\prec \dfrac{1+z}{1-z}+\dfrac{2z}{(1-z)^2},$$
	then
	$$\dfrac{f(z)}{z}\prec \dfrac{1+z}{1-z}.$$
\end{corollary}




\end{document}